\documentclass[a4paper, 12pt]{article}

\usepackage[margin=2cm]{geometry}
\usepackage{hyperref,xcolor}
\usepackage{amsthm, amssymb, amsmath}
\usepackage{enumerate}
\usepackage{algorithmic}
\usepackage{authblk}

\hypersetup{
    colorlinks=false,
	linkcolor=red,
	citecolor=red,
    pdfborder={0 0 0},
}

\title{Separation dimension of bounded degree graphs}

\author[1]{Noga~Alon\footnote{Supported  
by a USA-Israeli
BSF grant, by an ISF grant, by the Hermann Minkowski Minerva Center
for
Geometry at Tel Aviv University and by the Israeli I-Core
program.}}
\author[2]{Manu~Basavaraju}
\author[3]{L.~Sunil~Chandran}
\author[4]{Rogers~Mathew\footnote{Supported by VATAT Post-doctoral Fellowship, Council of Higher Education, Israel.}}
\author[4]{Deepak~Rajendraprasad\footnote{Supported by VATAT Post-doctoral Fellowship, Council of Higher Education, Israel.}}

\affil[1]{
	Sackler School of Mathematics and Blavatnik School of Computer Science, \authorcr
	Tel Aviv University, Tel Aviv 69978, Israel. \authorcr
	\texttt{nogaa@tau.ac.il}
}
\affil[2]{
	University of Bergen, Postboks 7800, NO-5020 Bergen.\authorcr
	\texttt{manu.basavaraju@ii.uib.no}
}
\affil[3]{
	Department of Computer Science and Automation, \authorcr 
	Indian Institute of Science, Bangalore, India - 560012. \authorcr
	\texttt{sunil@csa.iisc.ernet.in}
}
\affil[4]
{
	The Caesarea Rothschild Institute, Department of Computer Science, \authorcr
	University of Haifa, 31095, Haifa, Israel. \authorcr
	\texttt{\{rogersmathew, deepakmail\}@gmail.com}
}
\date{}

\theoremstyle{definition}
\newtheorem{definition}{Definition}

\theoremstyle{plain}
\newtheorem{theorem}{Theorem}
\newtheorem{lemma}[theorem]{Lemma}

\theoremstyle{remark}
\newtheorem*{remark}{Remark}

\newcommand{\sdim}{\pi}
\newcommand{\la}{\operatorname{la}}

\newcommand{\ceil}[1]{\left\lceil #1 \right\rceil}
\newcommand{\ilog}{\operatorname{log^{\star}}}

\def\into{\rightarrow}
\def\N{\mathbb{N}}
\def\R{\mathbb{R}}

\def\F{\mathcal{F}}

\begin{document}
\maketitle

\begin{abstract}
The {\em separation dimension} of a graph $G$ is the smallest natural
number $k$ for which the vertices of $G$ can be embedded in $\R^k$ such
that any pair of disjoint edges in $G$ can be separated by a hyperplane
normal to one of the axes. Equivalently, it is the smallest possible
cardinality of a family $\F$ of total orders of the vertices of $G$
such that for any two disjoint edges of $G$, there exists at least one
total order in $\F$ in which all the vertices in one edge precede those
in the other. In general, the maximum separation dimension of a graph on
$n$ vertices is $\Theta(\log n)$. In this article, we focus on bounded
degree graphs and show that the separation dimension of a graph with
maximum degree $d$ is at most $2^{9\ilog d} d$. We also demonstrate that
the above bound is nearly tight by showing that, for every $d$, almost
all $d$-regular graphs have separation dimension at least  $\ceil{d/2}$.
\vspace{1ex}

\end{abstract}




\section{Introduction}

Let $\sigma:U \into [n]$ be a permutation of  elements of an $n$-set $U$. For two disjoint subsets $A,B$ of $U$, we say  $A \prec_{\sigma} B$ when every element of $A$ precedes every element of $B$ in $\sigma$, i.e., $\sigma(a) < \sigma(b), ~\forall (a,b) \in A \times B$. 
We say that $\sigma$ {\em separates} $A$ and $B$ if either $A \prec_{\sigma} B$ or $B \prec_{\sigma} A$. We use $a \prec_{\sigma} b$ to denote $\{a\} \prec_{\sigma} \{b\}$. For two subsets $A, B$ of $U$, we say $A \preceq_{\sigma} B$ when $A \setminus B \prec_{\sigma} A\cap B \prec_{\sigma} B \setminus A$. 

Families of permutations which satisfy some type of ``separation'' properties have been long studied in combinatorics. One of the early examples of it is seen in the work of Ben Dushnik in 1950 where he introduced the notion of \emph{$k$-suitability} \cite{dushnik}. A family $\F$ of permutations of $[n]$ is \emph{$k$-suitable} if, for every $k$-set $A \subseteq [n]$ and for every $a \in A$, there exists a $\sigma \in \mathcal{F}$ such that $A \preceq_{\sigma} \{a\}$. Fishburn and Trotter, in 1992, defined the \emph{dimension} of a hypergraph on the vertex set $[n]$ to be the minimum size of a family $\F$ of permutations of $[n]$ such that every edge of the hypergraph is an intersection of \emph{initial segments} of $\F$ \cite{FishburnTrotter1992}. It is easy to see that an edge $e$ is an intersection of initial segments of $\F$ if and only if for every $v \in [n] \setminus e$, there exists a permutation $\sigma \in \F$ such that $e \prec_{\sigma} \{v\}$. Such families of permutations with small sizes have found applications in showing upper bounds for many combinatorial parameters like poset dimension \cite{kierstead1996order}, product dimension \cite{FurediPrague}, boxicity \cite{RogSunSiv} etc. Small families of permutations with certain separation and covering properties have found applications in event sequence testing \cite{chee2013sequence}.   

This paper is a part of our broad investigation\footnote{Many of our initial results on this topic are available as a preprint in arXiv \cite{basavaraju2012pairwise}.}  on a similar class of permutations which we make precise next.

\begin{definition}
\label{definitionPairwiseSuitable}
A family $\F$ of permutations of $V(H)$ is \emph{pairwise suitable} for a hypergraph $H$ if, for every two disjoint edges $e,f \in E(H)$, there exists a permutation $\sigma \in \F$  which separates $e$ and $f$. The cardinality of a smallest family of permutations that is pairwise suitable for $H$ is called the {\em separation dimension} of $H$ and is denoted by $\sdim(H)$. 
\end{definition}

A family $\F = \{\sigma_1, \ldots, \sigma_k\}$ of permutations
of a set $V$ can be seen as an embedding of $V$ into $\R^k$ with
the $i$-th coordinate of $v \in V$ being the rank of $v$ in 
$\sigma_i$. Similarly, given any embedding of $V$ in $\R^k$, we
can construct $k$ permutations by projecting the points onto each
of the $k$ axes and then reading them along the axis, breaking the
ties arbitrarily. From this, it is easy to see that $\sdim(H)$ is the
smallest natural number $k$ so that the vertices of $H$ can be embedded
into $\R^k$ such that any two disjoint edges of $H$ can be separated by
a hyperplane normal to one of the axes. This prompts us to call such an
embedding a {\em separating embedding} of $H$ and $\sdim(H)$ the {\em
separation dimension} of $H$.

A major motivation to study this notion of separation is its interesting connection with a certain well studied geometric representation of graphs. The \emph{boxicity} of a graph $G$ is the minimum natural number $k$ for which $G$ can be represented as an intersection graph of axis-parallel boxes in  $\R^k$.  It is established in \cite{basavaraju2014wg} that the separation dimension of a hypergraph $H$ is equal to the boxicity of the intersection graph of the edge set of $H$, i.e., the line graph of $H$.

It is easy to check that separation dimension is a monotone property,
i.e., adding more edges to a graph cannot decrease its separation
dimension. The separation dimension of a complete graph on $n$ vertices
and hence the maximum separation dimension of any graph on $n$ vertices
is $\Theta(\log n)$ \cite{basavaraju2014wg}. The separation dimension
of sparse graphs, i.e., graphs with linear number of edges, has also
been studied.
It is known that the maximum separation dimension of a $k$-degenerate
graph on $n$ vertices is $O(k \log\log n)$ and there exists a family of
$2$-degenerate graphs with separation dimension $\Omega(\log\log n)$
\cite{basavaraju2014icgt}. This tells us that even graphs in which
every subgraph has a linear
number of edges (a $2$-degenerate graph on $m$ vertices has at most $2m$
edges) can have unbounded separation dimension. Hence it is interesting to
see what additional sparsity condition(s) can ensure that the separation
dimension remains bounded.

Along that line of thought, here we investigate how large can the
separation dimension of bounded degree graphs be. To be precise, we
study the order of growth of the function $f : \N \into \N$ where $f(d)$
is the maximum separation dimension of a $d$-regular graph. Since any
graph $G$ with maximum degree $\Delta(G)$ at most $d$ is a subgraph
(not necessarily spanning) of a $d$-regular graph, and since separation
dimension is a monotone property, $\max\{\sdim(G) : \Delta(G) \leq d\}
= f(d)$. In this note, we show that for any $d$,
\[
\ceil{\frac{d}{2}} \leq f(d) \leq 2^{9 \ilog d} d.
\]
It is not difficult to improve the constant $9$ in the upper
estimate above, we make no attempt to optimize it here.

We arrive at the above upper bound using probabilistic methods
and it improves the $O(d \log\log d)$ bound which follows from
\cite{RogSunSiv}, once we note the connection between separation dimension
and boxicity established in \cite{basavaraju2014wg}. The upper bound in
\cite{RogSunSiv} was proved using $3$-suitability. The lower bound here
is established by showing that, for every $d$, almost all $d$-regular
graphs 
have separation dimension at least  $\ceil{d/2}$. A critical
ingredient of our proof is the small set expansion property of random
regular graphs. Prior to this, the best lower bound known for $f(d)$
was $\log(d)$ which is the separation dimension of $K_{d,d}$, the
$d$-regular complete bipartite graph. Since it is known 
(see \cite{basavaraju2014wg})
that $\sdim(G)
\in O(\chi_a(G))$ where $\chi_a(G)$ is the acyclic chromatic number
of $G$, which is the minimum number of colors in a proper vertex
coloring so that the union of any two color classes contains no
cycle,
and it is also known (\cite{alon1991acyclic}) 
that, for every $d$, almost all $d$-regular
graphs have acyclic chromatic number $O(d)$,
it follows that, for every $d$, almost all $d$-regular graphs have
separation dimension $\Theta(d)$. 
It seems plausible to conjecture that in fact
$f(d)=\Theta(d)$ but at the moment we are unable to prove or
disprove this conjecture.

\section{Upper bound}

In order to establish the upper bound on $f(d)$ (Theorem \ref{theoremUpperBound}), We need two technical lemmata, the first of which (Lemma \ref{lemmaMaxPairsParts}) we had established in \cite{basavaraju2014wg}, and the second one (Lemma \ref{lemmaLogDegreePartition}), which is similar to Lemma $4.2$ in \cite{FurediKahn}, is established using the Local Lemma.  

\begin{lemma}[\cite{basavaraju2014wg}, also Lemma $7$ in \cite{basavaraju2012pairwise}]
\label{lemmaMaxPairsParts}
Let $P_G=\{V_1, \ldots, V_r\}$ be a partitioning of the vertices of a graph $G$, i.e., $V(G) = V_1 \uplus \cdots \uplus V_r$. Let $\hat{\sdim}(P_G) = \max_{i,j \in [r]} \sdim(G[V_i \cup V_j])$. Then, $\sdim(G) \leq 13.68 \log r + \hat{\sdim}(P_G) r$.   
\end{lemma}

A useful consequence of Lemma \ref{lemmaMaxPairsParts} is that if we can somehow partition the vertices of a graph $G$ into $r$ parts such that the separation dimension of the union of any two parts is bounded, then $\sdim(G)$ is $O(r)$. For example, consider $P_G$ to be the partition of $V(G)$ corresponding to the color classes in a distance-two coloring of $G$, i.e, a vertex coloring of $G$ in which no two vertices of $G$ which are at a distance at most $2$ from each other are given the same color. Then the subgraphs induced by any pair of color classes is a collection of disjoint edges and hence $\hat{\pi}(P_G) \leq 1$. It is easy to see that a distance-two coloring of a $d$-regular graph can be done using $d^2 + 1$ colors and hence $f(d)$ is $O(d^2)$. Similarly, taking the parts to be color classes in an acyclic vertex coloring of $G$, it follows that $\sdim(G)$ is $O(\chi_a(G))$, where $\chi_a(G)$ denotes the acyclic chromatic number of $G$ and hence by \cite{alon1991acyclic}, $f(d)$ is $O(d^{4/3})$.  It was shown in \cite{RogSunSiv} that if $G$ is a linegraph of a multigraph then the boxicity of $G$ is at most $2 \Delta(\ceil{\log\log \Delta} + 3) + 1$ where $\Delta$ is the maximum degree of $G$. Since the separation dimension of a graph is equal to the boxicity of its linegraph \cite{basavaraju2014wg}, it follows that $f(d)$ is $O(d \log\log d)$. Theorem \ref{theoremUpperBound} improves this bound.

\begin{lemma}[The Lov\'asz Local Lemma, \cite{lovaszlocallemma}] 
\label{lemmaLovaszLocal}
Let $G$ be a graph on vertex set $[n]$ with maximum degree $d$ and let $A_1, \ldots , A_n$ be events defined on some probability space such that for each $i$, 
$Pr[A_i] \leq 1/4d$. Suppose further that each $A_i$ is jointly independent of the events $A_j$ for which $\{i,j\} \notin E(G)$. Then $Pr[\overline{A_1} \cap \cdots \cap \overline{A_n}] > 0$. 
\end{lemma}

\begin{lemma}
\label{lemmaLogDegreePartition}
For a graph $G$ with maximum degree $\Delta \geq 2^{64}$, there exists
a partition of $V(G)$  into $\ceil{ 400 \Delta / \log \Delta }$
parts such that for every vertex $v \in V(G)$ and for every part $V_i,
\, i \in \big[ \ceil{ 400 \Delta / \log \Delta } \big],\, \lvert N_G(v)
\cap V_i \rvert \leq \frac{1}{2} \log \Delta$.
\end{lemma}
\begin{proof}
Since we can have a $\Delta$-regular supergraph (with possibly more vertices) of $G$ we can as well assume that $G$ is $\Delta$-regular. Let $r = \ceil{ \frac{400 \Delta}{\log \Delta} } \leq \frac{401\Delta}{\log \Delta}$. Partition $V(G)$ into $V_1 , \ldots , V_r$ using the following procedure: for each $v \in V(G)$, independently assign $v$ to a set $V_i$ uniformly at random from $V_1, \ldots , V_r$. 

We use the following well known multiplicative form of the 
Chernoff Bound (see, e.g., Theorem A.1.15 in \cite{alon2008probabilistic}). 
Let $X$ be a sum of mutually independent indicator 
random variables with $\mu = E[X]$. Then for any $\delta > 0$,
$Pr[X \geq (1+\delta) \mu] \leq c_{\delta}^{\mu},$
where $c_{\delta} = e^{\delta} / (1 + \delta)^{(1 + \delta)}$. 

Let $d_i(v)$ be a random variable that denotes the number of neighbours of $v$ in $V_i$. Then $\mu_{i,v} = E[d_i(v)] = \frac{\Delta}{r} \leq \frac{1}{400} \log \Delta$. For each $v \in V(G), i \in [r]$, let $E_{i,v}$ denote the event $d_i(v) \geq \frac{1}{2}\log \Delta$. Then applying the above Chernoff bound with $\delta = 199$, we have $Pr[E_{i,v}] = Pr[d_i(v) \geq 200 \frac{\log \Delta}{400}] \leq 2^{-3.1 \log \Delta} = \Delta^{-3.1}$. In order to apply Lemma \ref{lemmaLovaszLocal}, we construct a graph $H$ whose vertex set is the collection of ``bad'' events $E_{i,v}$, $i \in [r], v \in V(G)$, and two vertices $E_{i,v}$ and $E_{i',v'}$ are adjacent if and only if the distance between $v$ and $v'$ in $G$ is at most $2$. Since for each $i \in [r]$ and $v \in V(G)$, the event $E_{i,v}$ depends only on where the neighbours of $v$ went to in the random partitioning, it is jointly independent of all the events $E_{i',v'}$ which are non-adjacent to it in $H$. It is easy to see that the maximum degree of $H$, denoted by $d_H$, is at most $(1 + \Delta + \Delta(\Delta-1))r = (1+\Delta^2)r \leq \frac{402 \Delta^3}{\log \Delta}$. For each $i \in [r], v \in V(G)$, $Pr[E_{i,v}] \leq \frac{1}{\Delta^{3.1}}  \leq \frac{\log \Delta}{1608 \Delta^3} \leq \frac{1}{4d_H}$. Therefore, by Lemma \ref{lemmaLovaszLocal}, we have $Pr[\bigcap_{i \in [r], v \in V(G)} \overline{E_{i,v}}] > 0$. Hence there exists a partition satisfying our requirements. 
\end{proof}

\begin{theorem}
For every positive integer $d$, $f(d) \leq 2^{9 \ilog d} d$.
\label{theoremUpperBound} 
\end{theorem}
\begin{proof}
If $d \leq 1$, then $G$ is a collection of matching edges and disjoint
vertices and therefore $f(1) = 1$.  When $d > 1$, it follows from
Theorem $10$ in \cite{RogSunSiv} that $\sdim(d) \leq (4d - 4)(\ceil{
\log \log (2d - 2) } + 3) + 1$. For every $1 < d < 2^{64}$, it can be
verified that 
$$
(4d - 4)(\ceil{ \log \log (2d - 2) } + 3) + 1 \leq 2^{9
\ilog d} d.
$$ 
Therefore, the statement of the theorem is true for every
$d < 2^{64}$.

For $d \geq 2^{64}$, let $P_G$ be a partition of $V(G)$ into $V_1
\uplus \cdots \uplus V_{r}$ where $r=\ceil{ 400 d / \log d }$ and
$|N_G(v) \cap V_i| \leq \frac{1}{2} \log d, ~ \forall v \in V(G),
i \in [r]$. Existence of such a partition is guaranteed by Lemma
\ref{lemmaLogDegreePartition}. From Lemma \ref{lemmaMaxPairsParts},
we have $\sdim(G) \leq 13.68 \log r +  \hat{\sdim}(P_G) r$ where
$\hat{\sdim}(P_G) = \max_{i,j \in [r]} \sdim(G[V_i \cup V_j])$. Since
$|N_G(v) \cap V_i| \leq \frac{1}{2} \log d$ for every $v \in V(G), i \in
[r]$, the maximum degree of the graph $G[V_i \cup V_j]$ is at most $\log
d$ for every $i, j \in [r]$. Therefore, $\hat{\sdim}(P_G) \leq f(\log
d)$. Thus we have, for every $d \geq 2^{64}$,
\begin{eqnarray}
\label{equationRecurrencePi}
f(d) 
	& \leq & \ceil{ \frac{400 d}{\log d} } f(\log d) + 13.68\log \ceil{ \frac{400 d}{\log d}} \nonumber \\
	& \leq & 2^9 \frac{d}{\log d} f(\log d).
\end{eqnarray}

Now we complete the proof by using induction on $d$. The statement is true for all value of $d < 2^{64}$  and we have the recurrence relation of Equation (\ref{equationRecurrencePi}) for larger values of $d$. For an arbitrary $d \geq 2^{64}$, we assume inductively that the bound in the statement of the theorem is true for all smaller values of $d$. Now since $d \geq 2^{64}$, we can apply the recurrence in Equation (\ref{equationRecurrencePi}). Therefore
\begin{eqnarray*}
\label{equationInductionPi}
f(d) 
	& \leq & 2^9 \frac{d}{\log d} f(\log d) \\
	& \leq & 2^9 \frac{d}{\log d} 2^{9 \ilog(\log d)} \log d, \, \textnormal{(by induction)} \\
	& = & 	 2^{9 \ilog d} d. 
\end{eqnarray*}

\end{proof}

\section{Lower bound}

To simplify the presentation we omit the floor  and ceiling signs
throughout the proof, whenever these are not crucial.
Consider the probability space of all labelled $d$-regular graphs on $n$ vertices with uniform distribution. We say that almost all $d$-regular graphs have some property $P$ if the probability that $P$ holds tends to $1$ as $n$ tends to $\infty$. 
We need the following well known 
fact about expansion of small sets in $d$-regular graphs.

\begin{theorem}[c.f., e.g., Theorem $4.16$ in \cite{hoory2006expander}]
\label{theoremSmallSetExpansion}
Let $d \geq 3$ be a fixed integer. Then for every $\delta > 0$, there exists $\epsilon > 0$ such that for almost every $d$-regular graph there are at most $(1 + \delta)|S|$ edges inside any $S \subset V(G)$ of size $\epsilon |V(G)|$ or less.
\end{theorem}

\begin{theorem}
\label{theoremLowerBound}
For every positive integer $d$, almost all $d$-regular graphs have separation dimension at least $\ceil{d/2}$.
\end{theorem}

\begin{proof}
The claim is easy to verify for $d = 1$ and $d = 2$ and hence we assume $d \geq 3$. Let $\delta \in (0, 1)$ be arbitrary and $\epsilon > 0$ chosen so as to satisfy the small set expansion guaranteed by Theorem \ref{theoremSmallSetExpansion}. Choose $n$ larger than $4(d+1)(1/\delta\epsilon)^{d/2}$ and let $G$ be a $d$-regular graph on $n$ vertices which 
satisfies the small set expansion property of Theorem \ref{theoremSmallSetExpansion}. Note that almost all $d$-regular graphs qualify.

For a permutation $\sigma = (v_0, \ldots, v_{n-1})$ of $V(G)$, we call an
edge $\{v_i,v_j\}$ of $G$ {\em $\sigma$-short} if $|i - j| \leq \delta
\epsilon n$. We claim that for any permutation $\sigma$, the number
of $\sigma$-short edges of $G$ is at most $\left(\frac{1 + \delta}{1-
\delta} \right) n$. To see this, cover the permutation $\sigma$ with
overlapping blocks of size $b = \epsilon n$ and amount of overlap $s =
\delta b = \delta \epsilon n$. To be precise, the blocks are
$$
B_i = \{v_j : i(1 - \delta)b \leq j < i(1 - \delta)b + b\},
\forall\; 0 \leq i < \frac{1}{\epsilon(1 - \delta)}.
$$ 
Due to the overlap between the blocks, every $\sigma$-short edge is inside some block $B_i$. Since each block $B_i$ has at most $\epsilon n$ vertices, there are at most $(1 + \delta) \epsilon n$ edges in each of them and hence at most $\left(\frac{1 + \delta}{1 - \delta} \right) n$ $\sigma$-short edges in total.

Suppose now for contradiction that $\sdim(G) < \ceil{d/2}$ and hence,
being an integer, $\sdim(G) \leq \left( \frac{d}{2} - \frac{1}{2}
\right)$. Let $\F$ be a set of $\sdim(G)$ permutations which is pairwise
suitable for $G$. Then by the above discussion at most $\left(\frac{d}{2}
- \frac{1}{2}\right) \left( \frac{1 + \delta}{1 - \delta} \right)
n$ edges of $G$ are $\sigma$-short for some permutation $\sigma \in
\F$. Let us call the remaining edges, that is those which are not short
in any permutation in $\F$, {\em long}. Since $\delta \in (0,1)$
was arbitrary, we can choose it small enough so that $\left(\frac{d}{2}
- \frac{1}{2}\right) \left( \frac{1 + \delta}{1 - \delta} \right)  \leq
\left( \frac{d}{2} - \frac{1}{4} \right)$. Thus we have at least $n/4$
long edges in $G$.

Since $G$ is $(d+1)$-edge colorable (by Vizing's Theorem) the edges of $G$
can be partitioned into $d+1$ matchings. By averaging, at least one of
the matchings has at least $n/4(d+1)$ long edges $L = \{e_1,
\ldots, e_{\ceil{n/4(d+1)}}\}$. We complete the proof by arguing that $\F$
cannot separate all the edges in $L$.

For each permutation $\sigma = (\sigma_1, \ldots, \sigma_n) \in \F$
construct a graph $H_{\sigma}$ on the set of vertices $L$ in which each
$e_i$ is adjacent to $e_j$ if and only if $e_i$ and $e_j$ are separated
by $\sigma$. We claim that the chromatic number of each $H_{\sigma}$
is at most $1/\delta \epsilon$. Indeed, for $s=\delta \epsilon n$, every
long edge goes over at least one of the points $\sigma_{s}, \sigma_{2s},
\ldots ,\sigma_{s/\delta \epsilon}$, and all the edges $e_i$ that go over
the point $js$ form an independent set in $H_{\sigma}$.  Since all the
pairs of edges $e_i$ have to be separated by some $\sigma$, we conclude
that the union of the graphs $H_{\sigma}$ over all $\sigma \in \F$ is
the complete graph on $L$. However, since the complete graph on $L$ has
chromatic number $|L|=\ceil{n/4(d+1)}$ we need $n/4(d+1) \leq (1/\delta
\epsilon)^{d/2}$ contradicting the choice of $n$. Therefore $\sdim(G)
\geq \ceil{d/2}$.
\end{proof}

\begin{remark}
The {\em linear arboricity} $\la(G)$ of a graph $G$ is the minimum
number of linear forests (disjoint union of paths) that can cover the
edges of $G$. We cannot hope for a lower bound on $\sdim(G)$ that is
bigger than $\la(G)$ using the above technique. This is because we can
write down one permutation for each linear forest so that every edge of
$G$ appears as an edge of length $1$ in one of these permutations. It
is known that the linear arboricity of a $d$-regular graph is at
most $d/2 + cd^{2/3}(\log d)^{1/3}$ for some absolute constant $c$
(\cite[p. 78]{alon2008probabilistic}, see also \cite{ATW}),
and it is equal to $\ceil{(d+1)/2}$
for almost all random $d$-regular graphs \cite{mcdiarmid1990linear}.
\end{remark}

\begin{small}

\end{small}

\end{document}